\documentclass[10pt]{amsart}

\usepackage{tikz}
\usepackage{mathrsfs}
\usepackage{amsthm}

\usepackage{amsmath,amssymb}
\usepackage[all,poly,knot]{xy}
\usepackage{hyperref}
\usepackage[center]{titlesec} 
\usepackage{amsfonts}
\usepackage{color}
\titleformat{\section}{\centering\Large\bfseries}{\arabic{section}.}{1em}{}
\titleformat{\subsection}{\centering\large\bfseries}{\arabic{section}.\arabic{subsection}.}{1em}{}

\newtheorem{thm}{Theorem}[section]

\newtheorem{cor}[thm]{{Corollary}}
\newtheorem{lem}[thm]{{Lemma}}

\newtheorem{prop}[thm]{Proposition}

\newtheorem{conj}[thm]{{Conjecture}}
\newtheorem{ques}[thm]{{Question}}

\newtheorem{rmk}[thm]{Remark}
\theoremstyle{definition}
\newtheorem{defcon}[thm]{Definition-Construction}
\newtheorem{setup}[thm]{Setup}
\newtheorem*{acknowledgement*}{Acknowledgements}

\theoremstyle{remark}
\newtheorem*{rmk*}{Remark}

\theoremstyle{definition}

\numberwithin{equation}{section}

\newcommand{\bZ}{{\mathbb Z}}

\newcommand{\mE}{{\mathcal E}}

\newcommand{\mU}{{\mathcal U}}

\newcommand{\nc}{\newcommand}
\nc{\on}{\operatorname}

\nc{\Aut}  {{\on{\mathrm  {Aut}}}}
\nc{\End}  {{\on{\mathrm  {End}}}}
\nc{\Fil}  {{\on{\mathrm  {Fil}}}}
\nc{\Frac} {{\on{\mathrm  {Frac}}}}
\nc{\Gal}  {{\on{\mathrm  {Gal}}}}
\nc{\GL}   {{\on{\mathrm  {GL}}}}
\nc{\Gr}   {{\on{\mathrm  {Gr}}}}
\nc{\Hom}  {{\on{\mathrm  {Hom}}}}
\nc{\id}   {{\on{\mathrm  {id}}}}
\nc{\PGL}  {{\on{\mathrm  {PGL}}}}
\nc{\rank} {{\on{\mathrm  {rank}}}}
\nc{\rmd}  {{\on{\mathrm  {d}}}}
\nc{\Spec} {{\on{\mathrm  {Spec}}}}

\nc{\HDF}  {{\on{\mathcal  {HDF}}}}
\nc{\HIG}  {{\on{\mathcal  {HIG}}}}
\nc{\IC}   {{\on{\mathcal {IC}}}}
\nc{\MCF}  {{\on{\mathcal {MCF}}}}
\nc{\MCFa} {{\on{\mathcal {MCF}_{[0,a]}}}}
\nc{\MF}   {{\on{\mathcal {MF}}}}
\nc{\MFa}  {{\on{         \MF_{[0,a]}}}}
\nc{\MFaf} {{\on{         \MF_{[0,a],f}}}}
\nc{\MIC}  {{\on{\mathcal {MIC}}}}
\nc{\MICa} {{\on{\mathcal {MIC}_{[0,a]}}}}
\nc{\THDF} {{\on{\mathcal {THDF}}}}
\nc{\THDFa}{{\on{\mathcal {THDF}_{[0,a]}}}}
\nc{\TMF}  {{\on{\mathcal {TMF}}}}
\nc{\TMFa} {{\on{\TMF_{[0,a]}}}}
\nc{\TMFaf}{{\on{\TMF_{[0,a],f}}}}
\nc{\tMIC} {{\on{\widetilde{\mathcal{MIC}}}}}
\nc{\FIsoc} {{\on{\textbf{F-Isoc}}}}
\def\fai{\varphi_{\lambda,\mathbb{F}_{p}}}

\def\nilp{\mathrm{nilp}}

\numberwithin{equation}{thm}

\usepackage[top=1.5in, bottom=1.5in, left=1in, right=1in]{geometry}

\begin{document}
\title{ Finiteness of logarithmic crystalline representations} 

\author{Raju Krishnamoorthy}
\email{raju@uga.edu}  
\address{Department of Mathematics, University of Georgia, Athens, GA 30605, USA}
\author{Jinbang Yang}
\email{yjb@mail.ustc.edu.cn}
\address{Institut f\"ur Mathematik, Universit\"at Mainz, Mainz 55099, Germany}
\author{Kang Zuo}
\email{zuok@uni-mainz.de}
\address{Institut f\"ur Mathematik, Universit\"at Mainz, Mainz 55099, Germany}

\maketitle
 	
\begin{abstract} Let $K$ be an unramified $p$-adic local field and let $W$ be the ring of integers of $K$. Let $(X,S)/W$ be a smooth proper scheme together with a normal crossings divisor. We show that there are only finitely many log crystalline $\bZ_{p^f}$-local systems over $X_K\setminus S_K$ of given rank and with geometrically absolutely irreducible residual representation, up to twisting by a character. The proof uses $p$-adic nonabelian Hodge theory and a finiteness result due Abe/Lafforgue.
\end{abstract}
 
%\tableofcontents 
	
\section{Introduction} 
To state our main theorem, the following setup will be convenient.
\begin{setup}\label{setup}Let $p$ be an odd prime, let $f\geq 1$ be a postive integer, and let $k$ be a finite field  containing $\mathbb F_{p^f}$. Let $W=W(k)$ be the ring of Witt vectors and set $K=\mathrm{Frac}(W)$. Let $(X,S)/W$ be a smooth proper scheme together with a relative normal crossings divisor over $W$. Denote by $U=X\setminus S$ and by $\mathcal U_K$ the rigid analytification of $U_K=U\times_{\Spec W} \Spec K$.
\end{setup}
\begin{thm}\label{mainthmRep} Notation as in Setup \ref{setup}, and fix a positive integer $r\leq p-2$. Then there are only finitely many isomorphism classes of logarithmic crystalline representations $\rho:\pi_1(U_K)\rightarrow \GL_r(\bZ_{p^f})$ whose residual representation is geometrically absolutely irreducible, up to twisting by a character of $\text{Gal}(\bar{K}/K)$.
\end{thm}
For clarity, the statement that the residual representation of $\rho$ is geometrically absolutely irreducible means that the composite representation $$\overline{\rho}\colon \pi_1(U_{\bar K})\rightarrow \pi_1(U)\rightarrow \GL_r(\bZ_{p^f})\rightarrow \GL_r(\mathbb{F}_{p^f})\rightarrow \GL_r(\bar{\mathbb F}_p)$$ of the geometric fundamental group is irreducible. Crystalline representations are a $p$-adic analog of polarized variations of Hodge structures. Therefore Theorem \ref{mainthmRep} is an arithmetic analog of a theorem of Deligne \cite{Del87}. See also the very recent work of Litt for a finiteness result in a different spirit \cite{Li18}.

The proof of Theorem \ref{mainthmRep} implicitly relies on the work of Abe on a global $p$-adic Langlands correspondence for smooth curves over finite fields, which itself follows the work of Lafforgue on the $l$-adic global Langlands correspondence for smooth curves over finite fields. We now describe the sections of this note.
\begin{itemize}
\item In Section \ref{section:preliminaries}, we explain the preliminary material.
\item In Section \ref{section:proof}, we prove Theorem \ref{mainthmRep} by reducing it to a statement about Higgs-de Rham flows.
\item In Section \ref{section:nonexample}, we show that Theorem \ref{mainthmRep} is false if $k\cong \overline{\mathbb F}_p$ in Setup \ref{setup}.
\item In Section \ref{section:uniform}, we speculate on a uniform upper bound of the number of crystalline representations.
\end{itemize}

\section{Preliminaries}\label{section:preliminaries}
We briefly describe the main players, with notation as in Setup \ref{setup}. Recall that a logarithmic Fontaine-Faltings module over $(X,S)$ is quadruple $(V,\nabla,\Fil,\Phi)$, where
\begin{itemize}
\item $V$ is a vector bundle on $X$;
\item $\nabla: V\rightarrow V\otimes \Omega_{X/W}^1(\log S)$ is an integrable connection on $X$ with logarithmic poles along $S$;
\item  $\Fil$ is a filtration of $V$ by subbundles that is locally split and satisfies Griffith transversality; and
\item $\Phi$ is a strongly divisible Frobenius structure that is horizontal with respect to $\nabla$.
\end{itemize}
There are several perspectives on the $\Phi$-structure. The original perspective, proposed by Faltings, is to work locally: one defines a strongly divisible $\Phi$ structure on $(V,\nabla,\Fil)$ over a small affine, whose $p$-adic completion admits a Frobenius lift, and then one glues \cite{Fal89}. A second view, via nonabelian Hodge theory, only works when the level of $(V,\nabla,\Fil,\Phi)$ is no greater than $p-2$, but it is a global description. The key to this perspective is the $p$-adic inverse Cartier transform defined by Lan-Sheng-Zuo~\cite{LSZ13a}; then $\Phi$ is an isomorphism 
\[\Phi: C^{-1}(\Gr(V),\Gr(\nabla),V,\nabla,\Fil) \overset{\sim}{\longrightarrow} (V,\nabla).\]
Finally, by forgetting the filtration, a logarithmic Fontaine-Fontaine module yields a logarithmic $F$-crystal in finite, locally free modules on the logarithmic crystalline site of $(X_k,S_k)/W$.

Let $\MF_{[0,a],f}^\nabla((X,S)/W)$ be the category of Fontaine-Faltings modules $\{(V,\nabla,\Fil,\Phi,\iota)\}$ over $(X,S)/W$ with Hodge-Tate weights in $[0,a]$ and endomorphism structure $\iota\colon W(\mathbb{F}_{p^f})\rightarrow \text{End}_{\MF}(V,\nabla,\Fil,\Phi)$. If $a\leq p-2$, then the fundamental work of Fontaine-Lafaille-Faltings constructs a fully faithful functor to the category of $\GL(\mathbb Z_{p^f})$ logarithmic crystalline local systems over $U_K$ and with unipotent local monodromy around $S_K$ \cite{Fal89}.\footnote{The nilpotence of the residues of the connection may be seen as follows: the Higgs field is nilpotent and the  de Rham bundle comes from inverse Cartier transform.}
Let $\FIsoc_{\nilp}^{\log}(X_k)$ be the category of convergent log-$F$-isocrystals on $(X_k,S_k)$ with nilpotent residues around $S_k$. Let $\FIsoc^\dagger(U_k)$ be the category of overconvergent $F$-isocrystals over $U_k$. These are both $\mathbb{Q}_p$-linear Tannakian categories.

Recall that a convergent $F$-isocrystal over $U_k = X_k\setminus S_k$ can be realized as an vector bundle $\mE$ on $\mU_K$ equipped with an integrable connection together with an isomorphism $\sigma^*\mE\simeq\mE$. By forgetting Hodge filtration, tensoring with $\mathbb Q_p$ and then  restricting  to $\mathcal U_K$, one obtains a functor
\[\psi_1\colon \MF_{[0,a],f}^\nabla((X,S)/W) \rightarrow \FIsoc^{\log}_{\nilp}(X_k)_{\mathbb{Q}_{p^f}}.\]
(The notation on the right hand side refers to the $\mathbb{Q}_{p^f}$-linearization of $\FIsoc^{\log}_{\nilp}(X_k)$.) By a fundamental theorem of Kedlaya, the forgetful functor 
 \[\psi_2\colon \FIsoc^{\log}_{\nilp}(X_k) \rightarrow \FIsoc^\dagger(U_k)\]
 is fully faithful  \cite[Proposition 6.3.2]{kedlayasemistableI}.
% \begin{thm}\label{Ked} Let $U_1$ be an open immersion of smooth $k$-variety $X_1$ such that $D_1=X_1\setminus U$ be a strict normal crossings divisor on $X_1$. Let $\mE$ be an isocrystal on $U_1$ overconvergent along $D_1$. Then $\mE$ has unipotent monodromy along $D_1$ if and only if $\mE$ extends to a convergent log-isocrystal with nilpotent residues on $(X_1, D_1)$. Moreover, the restriction functor, from convergent log-isocrystals with nilpotent residues on $(X_1, D_1)$ to isocrystals on $U_1$ overconvergent along $D_1$, is fully faithful.
%\end{thm}

For the definition of a \emph{tame} $F$-isocrystal, see \cite[1.2]{abeesnault}; in our case, an object $\mathcal{E}\in \FIsoc^{\dagger}(U_k)$ is tame if $\mathcal{E}$ extends to a logarithmic $F$-isocrystal on $(X_k,S_k)$. (In particular, there is no condition on the residues around $S_k$.) The following fundamental theorem follows from work of Abe and Lafforgue and is a consequence of the Langlands correspondence \cite{abelanglands,lafforguelanglands}.

%Recall that an $F$-isocrystal $\mE$ is call tame along $D$ if $\mE$ has $\bQ$-unipotent monodromy in the sense of~[Definition 1.3,Shiho 2010 Math. Ann.] along $D$.
\begin{thm}\label{finiteness} Let $k$ be a finite field of characteristic $p$, let $U_k$ be a smooth curve over $k$. Then there are finitely many isomorphism classes of irreducible tame objects of $\FIsoc^{\dagger}(U_k)_{\overline{\mathbb Q}_p}$ of bounded rank, up to twists by rank 1 objects. As a consequence, if $L/\mathbb{Q}_p$ is a \emph{finite} extension, then there are finitely many isomorphism classes of irreducible tame objects of $\FIsoc^{\dagger}(U_k)_{L}$ with finite order determinant.
\end{thm}

\section{The proof}\label{section:proof}

%\begin{lem}
%	Let $(V,\nabla,\Fil,\Phi)$ be a Fontaine-Faltings module of rank $r<p$. There there exists a rank $1$ $\phi$-module over $W$ such that. 
%\end{lem}

For any non-negative integer $a\leq p-2$, the category of logarthmic crystalline representations with Hodge-Tate weights in $[0,a]$ is equivalent to the category of logarithmic Fontaine-Faltings modules of level $\leq a$ via Faltings' $\mathbb{D}$-functor. To prove Theorem \ref{mainthmRep}, we first use a Lefschetz-style theorem to reduce to the curve case.
\begin{lem}\label{lemma:lefschetz}Notation as in Setup \ref{setup}. Then there exists a smooth projective relative curve $C\subset X$ over $W$ that intersects $S$ transversely, with the property that $\pi_1(C_K\cap U_K)\rightarrow \pi_1(U_K)$ is surjective. Therefore, to prove Theorem \ref{mainthmRep}, it suffices to consider the case when $X/W$ has relative dimension 1.
\end{lem}
\begin{proof}
We claim that there exists a smooth ample relative divisor $D\subset X$ over $W$ that intersects $S$ transversely. Indeed, pick some ample line bundle $L$ on $X$; then for all $m\gg 0$, the map $H^0(X,L^m)\rightarrow H^0(X_1,L^m_1)$ is surjective. On the other hand, for $m\gg 0$, the vector space $ H^0(X_1,L^m_1)$ has a section $s_1$ whose zero locus $V(s_1)$ is smooth and intersects $S_1$ transversely by Poonen's Bertini theorem \cite[Theorem 1.3]{poonenbertini}. Take any lift $s\in H^0(X,L^m)$ of $s_1$; then the zero locus $V(s)$ is smooth over $W$ and intersects $S$ transversely. Finally, it is well known that the map on fundamental groups $\pi_1(D_K\cap U_K)\rightarrow \pi_1(U_K)$ is surjective because $D_K\subset X_K$ is ample and $D_K$ intersects $S_K$ transversely (see \cite{EK15}). Proceed by induction. 

Now, as $\pi_1(C_K\cap U_K)\rightarrow \pi_1(U_K)$ is surjective, it follows that to prove Theorem \ref{mainthmRep}, it suffices to prove it for the pair $(C,S\cap C)$, i.e., we may reduce to the case of curves.
\end{proof}

We must now translate the property that the residual representation of $\rho$ is geometrically absolutely irreducible into a property about the objects of the associated Higgs-de Rham flow. This is accomplished by the following lemma.

\begin{lem}\label{stability} Notation as in Setup \ref{setup} and assume $X/W$ is a smooth curve. Let $\{(V,\nabla,\Fil,\Phi,\iota)_1\}$ be a logarithmic Fontaine-Faltings module over $(X,S)_1/k$ with Hodge-Tate weights in $[0,p-2]$ and endomorphism structure $\iota\colon \mathbb{F}_{p^f}\rightarrow \text{End}_{\MF}((V,\nabla,\Fil,\Phi)_1)$. Suppose the associated crystalline representation $\rho_1\colon \pi_1^{et}(X_K)\rightarrow \GL_r(\mathbb{F}_{p^f})$ is geometrically absolutely irreducible, i.e., the composite representation
	$$\overline{\rho}\colon \pi_1(X_{\bar K})\rightarrow \GL_r(\bar{\mathbb{F}}_p)$$
	is irreducible. Let 
	\begin{equation}\label{HDF_1} \tiny
	\xymatrix{  
		& (V,\nabla,\Fil)^{(0)}_1 \ar[dr] 
		&& (V,\nabla,\Fil)^{(1)}_1\ar[dr] 
		&& \cdots  \\
		(E,\theta)^{(0)}_1 \ar[ur]
		&& (E,\theta)^{(1)}_1 \ar[ur]
		&& (E,\theta)^{(2)}_1 \ar[ur]\\
	}
	\end{equation}		
	be the $f$-periodic Higgs-de Rham flow associated to $\{(V,\nabla,\Fil,\Phi,\iota)_1\}$.
	Then $(V,\nabla)^{(i)}_1$ and $(E,\theta)^{(i)}_1$ are all stable.
\end{lem}

\begin{proof} When $f=1$, this is \cite[Corollary 7.3]{LSZ13a}, but for completeness sake we write a proof. We first prove that $(E,\theta)$ is semistable and has vanishing rational Chern classes. This argument does not assume that $X/W$ is a curve.
 
 In the case that $S$ is empty, then as $(E,\theta)^{(0)}_1$ is periodic, it follows immediately from  \cite[Theorem 6.6]{LSZ13a} that $(E,\theta)^{(0)}_1$ is semistable with vanishing rational Chern classes. The same arguments works in the logarithmic case. Indeed, the Chern classes of all bundles in the flow rationally vanish: the key is that the inverse Cartier is locally a Frobenius pullback, and that the residues of $(V,\nabla)^{(0)}_1$ are nilpotent. In particular, the slope of $(E,\theta)^{(0)}_1$ vanishes. Then, exactly as in \cite[Theorem 6.6]{LSZ13a}, one shows that a slope $\lambda$ Higgs subsheaf of $(E,\theta)^{(0)}_1$ gives rise to a slope $p\lambda$ Higgs subsheaf of $(E,\theta)^{(1)}_1$. In particular, $(E,\theta)^{(0)}_1$ contains no Higgs subsheaf of positive slope; therefore $(E,\theta)^{(0)}_1$ is semistable with vanishing Chern classes.

Next, note that $(E,\theta)^{(0)}_1$ is stable if and only if $(V,\nabla)^{(0)}_1\cong C^{-1}(E,\theta)^{(0)}_1$ is stable. 

We now assume that $X/W$ is a curve. Let us show that $(E,\theta)^{(0)}_1$ is stable (the proof for other terms is precisely analogous). Suppose that $(E,\theta)^{(0)}_1$ is not stable. Then there is a proper Higgs subbundle $(E,\theta)'^{(0)}_1\subset (E,\theta)^{(0)}_1$ of slope $0$; note that $(E,\theta)'^{(0)}_1$ is \emph{automatically semistable}. By running the Higgs-de Rham flow starting with $(E,\theta)'^{(0)}_1$ (using the induced filtration), one obtains a sub  Higgs-de Rham flow
	\begin{equation}\label{eqn:subflow} \tiny
	\xymatrix{  
		& (V,\nabla,\Fil)'^{(0)}_1 \ar[dr] 
		&& (V,\nabla,\Fil)'^{(1)}_1\ar[dr] 
		&& \cdots  \\
		(E,\theta)'^{(0)}_1 \ar[ur]
		&& (E,\theta)'^{(1)}_1 \ar[ur]
		&& (E,\theta)'^{(2)}_1 \ar[ur]\\
	}
	\end{equation}
The initial bundle $E'^{(0)}_1$  has vanishing rational first Chern class. Therefore the same is true every bundle in Equation \ref{eqn:subflow}. There are only finitely many vector subbundles of $E^{(0)}_1$ with vanishing degree because $k$ is a finite field. Therefore Equation \ref{eqn:subflow} in fact initiates a \emph{preperiodic Higgs-de Rham flow}.
%	 As every Higgs-de Rham flow, where each term has vanishing rational Chern classes, over $(X_1,S_1)$ is preperiodic\footnote{This follows simply from the boundedness of the moduli space of logarithmic flat connections with vanishing rational Chern classes.}
The periodic part of Equation \ref{eqn:subflow} forms an $f'$-periodic sub Higgs-de Rham flow of (\ref{HDF_1}). (Note that $f\mid f'$ and $f'$ may be not equal to $f$.) Using the equivalence  of logarithmic periodic Higgs-de Rham flows and logarithmic Fontaine-Faltings modules \cite[Theorem 1.1]{LSYZ14}, together with the equivalence of logarithmic Fontaine-Faltings modules and torsion crystalline representations \cite[Theorem 2.6*, p. 41, i]{Fal89}, one deduces that the periodic part of Equation \ref{eqn:subflow} induces a representation
\[\rho'\colon \pi_1(X_{K(\zeta_{p^{f'}-1})}) \rightarrow \mathrm{GL}_{s}(\mathbb F_{p^{f'}})\]
where $s = \rank E'^{(0)}_1$ and $\zeta_{p^{f'}-1}$ a primitive $(p^{f'}-1)$-th root of unity. The inclusion map between these two periodic Higgs de Rham flows induces an inclusion of representations
\[\rho'\hookrightarrow \rho\mid_{\pi_1(X_{K(\zeta_{p^{f'}-1})})}\]
In particular one gets a proper sub-representation of $\overline{\rho}$. This contradicts the irreducibility of $\overline{\rho}$. 
\end{proof}

\begin{defcon}Notation as in Setup \ref{setup}. Let $m_0,\cdots,m_{f-1}\in \mathbb Z$. Define $\mathcal L_{\lambda}(k_0,k_1,\cdots,k_{f-1})$ to be the following flow over one point $\mathrm{Spec}(W)$ 
\begin{equation} \tiny
\xymatrix{  
	& (W,\Fil_0) \ar[dr] 
	&& (W,\Fil_1) \ar[dr] 
	&& \cdots \ar[dr] &  \\
	W \ar[ur]
	&& W \ar[ur]
	&& W \ar[ur]
	&& W \ar@/^20pt/[llllll]^{\varphi = \lambda}\\
	&&\\
}
\end{equation}
where $\Fil_i^{m_i} W= W$ and $\Fil_i^{m_i+1} W= 0$.
\end{defcon}
The data of $\mathcal L_{\lambda}(m_0,m_1,\cdots,m_{f-1})$ is equivalent to a \emph{rank $1$ filtered $\varphi^f$-module over $W$}, a.k.a. a filtered $\varphi$-module of rank $f$ over $W$ equipped with an endomorphism structure of $\mathbb{Z}_{p^f}$. By the Fontaine-Lafaille correspondence, this corresponds to a crystalline character $\mathrm{Gal}(\overline{K}/K)\rightarrow \mathbb \GL_1(\mathbb Z_{p^f})$.

Therefore, using Lemma Lemma \ref{lemma:lefschetz} and \ref{stability}, to prove Theorem \ref{mainthmRep}, it suffices to prove the following.
\begin{thm}\label{mainthmFM}
	Notation as in Setup \ref{setup}, with $X/W$ a curve, and fix $r\leq p-2$. Then there are only finitely many isomorphism classes of $f$-periodic Higgs-de Rham flow over $(X,S)/W$ of rank $r$ such that the reduction modulo $p$ of terms appeared in the flow are stable, up to twisting by a rank $1$ filtered $\varphi^f$-module over $W$.
\end{thm}

To prove Theorem \ref{mainthmFM}, we first record the following lemma, which follows immediately from \cite[Theorem 5.4]{Lan14}. 
\begin{lem}\label{fromQtoZ} Notation as in Setup \ref{setup}. Let $(V,\nabla)$ and $(V,\nabla)'$ be logarithmic flat connections over $(X,S)$ that each have stable reduction modulo $p$. If $(V,\nabla)_{\mathbb Q}$ and $(V,\nabla)'_{\mathbb Q}$ are isomorphic on $X_{\mathbb Q}$, then $(V,\nabla)$ and  $(V,\nabla)'$ are isomorphic as logarithmic flat connections on $(X,S)$. Moreover, if $f\colon (V,\nabla)\rightarrow (V,\nabla)'$ is a morphism of logarithmic flat connections that is an isomorphism modulo $p$, then $f$ is an isomorphism.
\end{lem}

%\begin{proof}Pick an isomorphism $f\colon (V,\nabla)\rightarrow (V,\nabla)'$. Since $(V,\nabla)_{\mathbb Q}\cong (V,\nabla)'_{\mathbb Q}$ one has $\mu(V/pV) = \mu(V'/pV')$. \raju{I do not understand this last sentence; can you please explain? Alternatively, we can simply add in the assumption that the Chern classes rationally vanish on the generic fiber and use the argument in our Lefschetz paper to show that means the Chern classes vanish on the special fiber.} Let $n$ be the least integer such that $p^nf(V) \subset V'$. Then $p^nf\pmod{p} \neq 0$ which is nonzero morphism from $V/pV$ to $V'/pV'$. Since both $(V,\nabla)/p(V,\nabla)$ and $(V,\nabla)'/p(V,\nabla)V'$ are stable of the same slope $p^nf\pmod{p}\colon (V,\nabla)/p(V,\nabla) \rightarrow (V,\nabla)'/p(V,\nabla)'$ is an isomorphism. By Nakayama lemma's $p^nf \colon (V,\nabla)\rightarrow (V,\nabla)'$ is also an isomorphism.
%\end{proof}

We recall the following lemma, which is due independently to Lan-Sheng-Zuo \cite[Lemma 7.1]{LSZ13a} and Langer \cite[Corollary 5.6]{Lan14}.
%\raju{I add the reference to Langer because it explicitly includes the logarithmic case.} 
\begin{lem}\label{lem:LSZ_Fil} Let $(Y,D)$ be a smooth projective variety together with a simply normal crossings divisor over an algebraically closed field $k$ and let $(V,\nabla)$ be a logarithmic flat bundle over $(Y,D)$. If there exists a Griffiths transverse filtration $\Fil$ on $(V,\nabla)$ such that the associated graded logarithmic Higgs module $(E,\theta)$ is Higgs stable, then $\Fil$ is unique up to a shift of index.
\end{lem}

\begin{lem}\label{uniqueFil}Notation as in Setup \ref{setup}. Let $\{(V,\nabla,\Fil,\Phi,\iota)\}$ and $\{(V,\nabla,\Fil,\Phi,\iota)'\}$ be two logarithmic Fontaine-Faltings modules over $(X,S)/W$ with Hodge-Tate weights in $[0,p-2]$ and endomorphism structure of $\mathbb{Z}_{p^f}$. Let 
		
\begin{equation}\label{HDF} \tiny
		\xymatrix{  
			& (V,\nabla,\Fil)^{(0)} \ar[dr] 
			&& (V,\nabla,\Fil)^{(1)} \ar[dr] 
			&& \cdots \ar[dr] &  \\
			(E,\theta)^{(0)} \ar[ur]
			&& (E,\theta)^{(1)} \ar[ur]
			&& (E,\theta)^{(2)} \ar[ur]
			&& (E,\theta)^{(f)} \ar@/^20pt/[llllll]^{\varphi}\\
			&&\\
		}
\end{equation}	
	
be the $f$-periodic Higgs-de Rham flow associated to $\{(V,\nabla,\Fil,\Phi,\iota)\}$ and 
\begin{equation}\label{HDF'} \tiny
		\xymatrix{  
			& (V,\nabla,\Fil)'^{(0)} \ar[dr] 
			&& (V,\nabla,\Fil)'^{(1)} \ar[dr] 
			&& \cdots\ar[dr] &  \\
			(E,\theta)'^{(0)} \ar[ur]
			&& (E,\theta)'^{(1)} \ar[ur]
			&& (E,\theta)'^{(2)} \ar[ur]
			&& (E,\theta)'^{(f)} \ar@/^20pt/[llllll]^{\varphi'}\\
			&&\\
		}
\end{equation}		
be the $f$-periodic Higgs-de Rham flow associated to $\{(V,\nabla,\Fil,\Phi,\iota)'\}$.
Assume the reduction modulo $p$ of all terms in the two flows are stable and assume there exists a morphism of de Rham bundle
\[g^{(0)}\colon (V,\nabla)^{(0)}\rightarrow (V,\nabla)'^{(0)}\] 
whose reduction modulo $p$ is nontrivial. Then $g^{(0)}$ is an isomorphism and there exists unique isomorphisms, up to scale, of de Rham bundles 
	\[g^{(i)} \colon (V,\nabla)^{(i)} \rightarrow (V',\nabla')^{(i)}\quad \text{for all } i = 0,1,\cdots,f-1.\]
Moreover, under these isomorphisms the Hodge filtration $\Fil^{(i)}$ coincides with $\Fil'^{(i)}$ after a shift of index and there exists $\lambda\in W^\times$ such that
\[\varphi =\lambda\varphi'.\]
\end{lem}
\begin{proof} After shifting the filtrations, we may assume 
	\[\Fil^0V^{(0)} = V^{(0)}, \quad  \Fil^1V^{(0)} \neq V^{(0)}, \quad  \Fil'^0V'^{(0)} = V'^{(0)}, \text{ and } \Fil'^1V'^{(0)} \neq V'^{(0)}.\]	
By Lemma \ref{fromQtoZ}, it follows that $g^{(0)}$ is an isomorphism and is unique up to scale. Lemma \ref{lem:LSZ_Fil} implies that 
\[g^{(0)}(\Fil\mid_{V^{(0)}})\pmod{p} = \Fil'\mid_{V'^{(0)}}\pmod{p}.\]
By following the proof of \cite[Proposition 7.5]{LSZ13a} verbatim, one deduces that $g^{(0)}(\Fil\mid_{V^{(0)}}) = \Fil'\mid_{V'^{(0)}}$. Then denote 
\[g^{(1)}:= C^{-1}\circ \Gr(g^{(0)}).\]
Since $g^{(0)}$ is an isomorphism, $g^{(1)}\colon (V,\nabla)^{(1)} \rightarrow (V,\nabla)'^{(1)}$ is also an isomorphism between de Rham bundles. By the stablility of the modulo $p$ reductions, this isomorphism is unique up to a scale. Inductively on the index $i$, one constructs isomorphisms $g^{(2)},g^{(3)},\cdots$. 

That the periodicity maps $\varphi$ and $\varphi'$ differ by an invertible element in $W$ follows from the stability.
\end{proof}

\begin{prop}\label{UniqueUptoTwist}
%Let $(V,\nabla,\Fil,\Phi,\iota)$ and $(V,\nabla,\Fil,\Phi,\iota)'$, $(V,\nabla)^{(i)}$, $(E,\theta)^{(i)}$, $(V,\nabla)'^{(i)}$ and $(E,\theta)'^{(i)}$ are given as in
Notation as in Lemma~\ref{uniqueFil}. Assume the reduction modulo $p$ of all terms are stable. If the $F$-isocrystals associated to $(V,\nabla,\Fil,\Phi,\iota)$ and $(V,\nabla,\Fil,\Phi,\iota)'$ are isomorphic, then there exists an integer $t$ in $\{0,1,\cdots,f-1\}$ and a filtered $\varphi^f$-module $\mathcal L$ over $W$ of rank 1 such that
\[(V,\nabla,\Fil,\Phi,\iota) \cong (V',\nabla',\Fil',\Phi',\sigma^t(\iota'))\otimes \mathcal L\]
\end{prop}
\begin{proof} First of all, the statement that the underlying $F$-isocrystals are isomorphic is the statement that $(V,\nabla,\Phi)_{\mathbb{Q}}\cong (V',\nabla',\Phi')_{\mathbb Q}$ in $\FIsoc(U_k)$. After shifting the filtration, we may assume 
\[\Fil^0V = V, \quad  \Fil^1V \neq V, \quad  \Fil'^0V' = V', \text{ and } \Fil'^1V' \neq V'.\]
Consider the composition maps
\[(V,\nabla)^{(0)}_{\mathbb Q} \hookrightarrow (V,\nabla)_{\mathbb Q} \cong (V,\nabla)'_{\mathbb Q} \twoheadrightarrow  (V,\nabla)'^{(i)}_{\mathbb Q}.\]
There is at least one index $i$ such that the composition map is not equal to zero. By shifting the index (equivalent to changing the endomorphism structure by a conjugation $\sigma^t$), we may assume $i=0$ and denote the composition map by 
\[g^{(0)}\colon (V,\nabla)^{(0)}_{\mathbb Q}\rightarrow (V,\nabla)'^{(0)}_{\mathbb Q}\]
 Multiplying by a suitable power of $p$, we may assume \[g^{(0)}(V^{(0)}) \subset V'^{(0)} \quad \text{and } g^{(0)}(V^{(0)}) \not\subset pV'^{(0)}.\]
 Then $g^{(0)}\pmod{p}$ is a non-trivial morphism between two stable de Rham bundles with the same slope. Thus $g^{(0)}\pmod{p}$ is an isomorphism. By Lemma~\ref{uniqueFil}, $g^{(0)}$ is an isomorphism and there exist isomorphisms
 \[g^{(i)}\colon (V,\nabla)^{(i)} \rightarrow (V,\nabla)'^{(i)}\] 
such that, up to a shift of filtration one has 
\[g^{(i)}(\Fil\mid_{V^{(i)}}) = \Fil'\mid_{V'^{(i)}}.\]

Now, we identify $(V,\nabla,\Fil)^{(i)}$ and $(V',\nabla',\Fil')^{(i)}$ via $g^{(i)}$. Both $(V,\nabla)^{(0)}$ and $(V',\nabla')^{(0)}$ are crystals on the logarthmic crystalline site of $(X_k,S_k)/W$. We may then consider both $\Phi^f$ and $\Phi'^f$ as morphisms
\[(\text{Frob}^*)^f(V,\nabla)^{(0)}\rightarrow (V,\nabla)^{(0)}\]
in the category of logarithmic crystals on $(X_k,S_k)/W$. As stability is an open condition, the logarithmic flat connection $(V,\nabla)^{(0)}_{\mathbb Q}$ is stable. As $\Phi'_{\mathbb Q}$ and $\Phi_{\mathbb{Q}}$ are isomorphisms, it follows that there exists $\lambda\in W^\times$ such that $\Phi = \lambda\Phi'$. Thus there exists an integer $t$ in $\{0,1,\cdots,f-1\}$ and a filtered $\varphi^f$-module $\mathcal L$ over $W$ of rank 1 such that
\[(V,\nabla,\Fil,\Phi,\iota) \cong (V',\nabla',\Fil',\Phi',\sigma^t(\iota'))\otimes \mathcal L.\]
\end{proof}

\begin{proof}[Proof of Theorem~\ref{mainthmFM}] 
Let $HDF$ and $HDF'$ be $f$-periodic Higgs-de Rham flows on $(X,S)$, whose terms modulo $p$ are all stable. Suppose the associated $F$-isocrystals are isomorphic. (By  \cite[Proposition 6.3.2]{kedlayasemistableI}, it does not matter if we consider these as objects of $\FIsoc_{\nilp}^{\log}(X_k)$ or $\FIsoc^{\dagger}(U_k)$.)  Then Proposition~\ref{UniqueUptoTwist} implies that $HDF$ and $HDF'$ differ by a twist. On the other hand, there are only finitely many isomorphism classes of overconvergent $F$-isocrystals on $U_k$, up to twist, by Theorem~\ref{finiteness}. The result follows.
%
%Consider the composite functor
%  \[\MF_{[0,a],f}^\nabla((X,S)/W)\otimes \mathbb{Q}_p \rightarrow \FIsoc_{\nilp}^{\log}(X_k) \rightarrow \FIsoc^{\dagger}(U_k),\]
%  where the second functor is fully faithful by. By , the theorem follows .	
\end{proof}

\section{Theorem \ref{mainthmRep} is false over $k=\bar{\mathbb{F}}_p$}\label{section:nonexample}\label{section:counterexample}
%There exists infinitely many $\text{GL}_2$ residue geometrically absolutely irreducible local systems over $(X,D)/\mathbb{Q}_p^{ur}$
%\\
Let $\lambda\in W$ with $\lambda\not\equiv 0,1\pmod{p}$. Let $(X,S)=(\mathbb P^1,\{0,1,\infty,\lambda\})$. Let $\mathcal M_{\lambda}$ denote the moduli space of semi-stable graded logarithmic Higgs bundles  $(E,\theta)$  of rank $2$ and degree $1$ over $(X,S)$. A Higgs bundle in this moduli space may be written as
\[(E,\theta)=(\mathcal O\oplus\mathcal O(-1),\, \theta\colon \mathcal O\xrightarrow{\theta}\mathcal{O}(-1)\otimes \Omega^1_X(\log S))\quad (*).\]
We attach a parabolic structure at one of the four points of $D$ with parabolic weight
$({1\over 2},\, {1\over 2})$, so that $(E,\theta)$ is parabolically stable and of parabolic degree zero. Let $\mathcal M^{par}_\lambda$ denote the moduli space of those type logarithmic Higgs bundles with a parabolic structure at
one of those four cusps.
%In the following  by applying Theorem 3.1 we first   just forget  the parabolic structure, run twisted Higgs de Rham flow for $(E,\theta)$ on $(X,D)$  and construct a residue geometrically absolute irreducible $\text{PGL}_2$-crystalline  local system $\mathbb L$ on $(X,D)/\mathbb{Q}_p^{ur}$. By taking the double cover $\pi: (X', D')\to (X,D)$ ramified on the parabolic point and another point from $D$, then the pulled back local system $\pi^*\mathbb L$ lifts to a  residue geometrically absolute irreducible
%$\text{GL}_2$-local system.
One defines an isomorphism
\[\mathcal M_{\lambda}\simeq \mathbb P^1_{W(k)}\]
by sending $(E,\theta)$ to the zero locus 
$(\theta)_0\in \mathbb P^1$. 
Consider the self map $\fai=\mathrm{Gr}\circ \mathcal C_{1,2}^{-1}$ on $\mathcal M_{\lambda}\otimes_Wk\simeq \mathbb P^1_{k}$ induced by Higgs-de Rham flow ($\otimes\mathcal O_{\mathbb P^1}((1-p)/2)$). Since $C^{-1}_{1,2}$ is a factor of the composition, $\fai$ factors through the Frobenius map, i.e., there exists a rational function $\psi_{\lambda}\in k(z)$ such that $\fai(z) = \psi_\lambda(z^p)$. In this note we call $\psi_\lambda$ the \emph{Verschiebung part} of the self map $\fai$. The periodic points of the self map $\fai$ are naturally corresponding to the twisted periodic Higgs-de Rham flows on $(X,S)_1$; one forgets the parabolic structure and simply runs a twisted Higgs-de Rham flow \cite[Section 4]{SYZ17}.

Conjecturally this self map is related to the muliplication by $p$ map on the elliptic curve
\[ C_\lambda\colon y^2 = x(x-1)(x-\lambda).\]
\begin{conj}[Conjecture 5.8 in \cite{SYZ17}]\label{ConjSYZ} The following diagram commutes
\begin{equation}
\xymatrix{C_\lambda \ar[r]^{[p]} \ar[d]^{\pi}& C_\lambda\ar[d]^{\pi}\\
\mathbb P^1 \ar[r]^{\fai}  &  \mathbb P^1 \\}
\end{equation}
\end{conj}

Let $(E,\theta)_{n}\in \mathcal M_{\lambda}(W_n)$ be a periodic Higgs bundle over $(X,S)_n$. Consider the self map $\Gr\circ C^{-1}$ on the deformation space $\mathrm{Def}_{(E,\theta)_{n}}(W_{n+1})$. Since $\mathrm{Def}_{(E,\theta)_{n}}(W_{n+1})$ is an $\mathbb A^1$-torsor space, we may identify a self map on $\mathrm{Def}_{(E,\theta)_{n}}(W_{n+1})$ with a self map on $\mathbb A^1$. Under this identification, $\Gr\circ C^{-1}$ is just a polynomial.

\begin{thm}[Sun-Yang-Zuo]\label{thm:infinitely_many_pgl} Notation as above. Then
\begin{enumerate}
\item There exists exactly $p^{2f}+1$ geometrically absolutely irreducible crystalline  $\mathrm{PGL}_2(\mathbb F_{p^f})$-local systems on $(X,S)$ that correspond to twisted $f$-periodic Higgs bundles from $\mathcal M_\lambda$.
\item Let $(E,\theta)_{n}\in \mathcal M_\lambda(W_n(\mathbb F_{q^{h}}))$ be a periodic Higgs bundle over $(X,S)_n$.  Then the polynomial associated to the self map $\Gr\circ C^{-1}$ on $\mathrm{Def}_{(E,\theta)_{n}}(W_{n+1})$ is
\[ \mathbb A^1(\mathbb F_q) \to  \mathbb A^1(\mathbb F_q),\quad z\to a\cdot z^p+b,\]
where $a, b \in \mathbb F_q$. Consequently if $a\neq 0$, then by solving the Artin-Schreier equation $az^p-z+b=0$ one obtains $p$ twisted $1$-periodic liftings over $W_{n+1}(\mathbb F_{q^{ph}})$. Moreover the constant $a$ is the derivative of Verschiebung part of $\fai$ at the point associated to $(E,\theta)_{n}\mid_{(X,S)_1}$. In particular the value of $a$ only depends on the reduction modulo $p$ of $(E,\theta)_{n}$.
\item  For $p\leq 50$, the Conjecture~\ref{ConjSYZ} holds. In this case, if the torsion point associated to $(E,\theta)_{n}\mid_{(X,S)_1}$ is not of order $2$, then the coefficient $a\neq 0$ if and only if the associated elliptic curve $C_{\bar{\lambda}}$ is not supersingular.
\end{enumerate}
\end{thm}
Let $\pi\colon (X',S')\rightarrow (X,S)$ be the double cover that is ramified at the parabolic point and one other point.  Using \cite[Theorem 4.6]{SYZ17} together with Theorem \ref{thm:infinitely_many_pgl}, one obtains the following corollary.
\begin{cor}Suppose $p\leq 50$ and the elliptic curve $C_{\bar{\lambda}}$ is not supersingular. Then there exists infinitely many log crystalline $\GL_2(\mathbb{Z}_p)$ local systems on $(X',S')_{\mathbb Q_p^{\text{unr}}}$ whose residual representation is absolutely geometrically irreducible.
\end{cor}

We end this section with a final remark.
\begin{rmk}
	In view of Theorem~\ref{mainthmRep} on finiteness if $a\neq0$ then the solutions of the Artin-Schreier equation in $(2)$ must lie in $\mathbb F_{q^{ph}}$ but not in $\mathbb F_{q^h}$ in almost all lifting steps. But, it seems difficult to prove that directly!
\end{rmk}

\section{Some speculations on a uniform upper bound}\label{section:uniform}
In Section~\ref{section:nonexample} we made an identification 
\[M_\lambda \simeq \mathbb P^1\]
by sending $(E,\theta)$ to the zero $\theta_0$ of the Higgs field $\theta$ and let $\pi\colon C_\lambda\rightarrow X$ be the associated double cover of $X=\mathbb P^1$, branched along $S$.  
\begin{conj}[Sun-Yang-Zuo]\label{varConjSYZ} A Higgs bundle in $\mathcal M_\lambda$ over a finite unramified extension is twisted $f$-periodic if and only if $\pi^*(\theta_0)$ is a  $(p^f\pm 1)$-torsion point in $C_\lambda.$
\end{conj}
Conjecture~\ref{varConjSYZ} is a consequence  of Conjecture~\ref{ConjSYZ}. It imples that the number of $\mathbb P_2(\mathbb Z_{p^f})$-crystalline local systems over $(X,S)$ over finite unramified extensions of $\mathbb Q_p$ is exactly $p^{2f}+1$. Conjecture~\ref{varConjSYZ} has been checked for the following cases.
\begin{enumerate}
\item When we only work modulo $p$ and $p\leq 50$.
\item When $C_{\bar{\lambda}}$ is supersingular and $p\leq 50$.
\item For all $p$, when the torsion point has order $1$, $2$, $3$, $4$ and $6$.
\end{enumerate}
\begin{rmk} When $C_{\bar{\lambda}}$ is supersingular, any $\text{GL}_2$-crystalline local system corresponding to Higgs bundles in $\mathcal M_\lambda(\mathbb Q_p^{ur})$ automatically descends to a local system
over a \emph{finite, unramified} extensions of $\mathbb{Q}_p$.  This contrasts with the situation when $C_{\bar{\lambda}}$ is an ordinary elliptic curve; we expect that most $\text{GL}_2$-crystalline local systems corresponding to Higgs bundles in $\mathcal M_\lambda(\mathbb Q_p^{ur})$ over $(X',S')/\mathbb{Q}_p^{ur}$ do not descend to  finite unramified extension of $\mathbb{Q}_p$.
\end{rmk}
We end by posing a conjecture, the first part of which is in the spirit of the Fontaine-Mazur conjecture and the second of which is analogous to a theorem of Litt \cite{Li18}.
\begin{conj} Let $(X,S)$ be a log pair over $W(\mathbb F_q)$ with $p^f\mid q$.
\begin{enumerate}
\item Let $\mathbb L_1$ and $\mathbb L_2$ be two $\mathbb Z_{p^f}$-crystalline local systems 
 over $(X,S)/W(\mathbb{F}_q)$. If $\mathbb{L}_1\simeq \mathbb{L}_2$ mod $p$, then $\mathbb L_1\simeq\mathbb L_2$.
 \item The number of isomorphism classes $\GL_r(\mathbb{Z}_{p^f})$-crystalline local systems over $(X,S)_{W(\mathbb{F}_{q^h})}$, as we let $h$ range through the positive integers, is finite.
\end{enumerate}
\end{conj}

\section{A brief discussion on the number of $\GL_2(\mathbb{Z}_p)$-local systems on $\mathbb{P}^1$ minus four points over a finite unramified extension of $\mathbb{Q}_p$ and with eigenvalues -1 around one of the four puncture points.}
Maintain notation as in Section \ref{section:nonexample}; in particular, $(X,S)=(\mathbb P^1,\{0,1,\infty,\lambda\})$. It seems possible that all of the local systems in Section \ref{section:nonexample}
could come from families of abelian varieties of Hilbert modular type over  $(X,S)/W(\mathbb{F}_q)$.
 We also conjecture that they correspond to  $(p\pm1)$-torsion points on the elliptic curve $C_\lambda/W(\mathbb{F}_q)$.

In fact in a joint paper with Lu-Lv-Sun-Yang-Zuo, we have checked this conjecture for torsion-points of orders 1,\, 2,\, 3,\, 4 and 6, when $K$ is a number field. There exists exactly 26 elliptic curves of $(X,S)/\mathcal{O}_K$ such that  the arithmetic local systems attached to those families correspond to arithmetic 1-periodic Higgs bundles from $\mathcal{M}_\lambda/\mathcal{O}_K$ and the zero locus of the Higgs fields are torsion points of order 1,\, 2,\, 3,\, 4 and 6. in $C_\lambda/\mathcal{O}_K$ under the pull back of $\pi.$

Via the techniques of Sections \ref{section:preliminaries} and \ref{section:proof} together with the theory of $p$-to-$l$ companions (due to Abe),  one obtains an inclusion
\begin{equation}\label{eqn:periodic_to_l-adic} (\mathcal{M}_\lambda)^{\text{periodic}}\hookrightarrow \mathcal{M}^{\ell\text{-adic}}_\lambda.\end{equation}
Here, $(\mathcal{M}_\lambda)^{\text{periodic}}$ consists of those Higgs bundles in $\mathcal{M}_{\lambda}$ which are periodic (without specifying the periodicity map) and $\mathcal{M}^{\ell\text{-adic}}_\lambda$ is the set of equivalence classses of (geometrically) irreducible systems $\text{GL}_2(\bar {\mathbb Q}_\ell)$ -local systems over $(X_k,S_k)$ with prescribed local monodromy on the cusp points, up to twisting by a character on $\mathbb{F}_q$. Indeed, given a $p$-adically periodic Higgs bundle in $\mathcal M_{\lambda}$, we may forget the filtration to obtain an overconvergent $F$-isocrystal $\mathcal E$ on $U_k$; moreover, this gives an injective map from the set of equivalence classes of such Higgs bundles (without specifying the periodicity map) to the set of isomorphism classes of (overconvergent) $F$-isocrystals up to twisting by a character on $\mathbb{F}_q$ by Proposition \ref{UniqueUptoTwist}. Picking a field isomorphism $\sigma\colon \overline{\mathbb Q}_p\rightarrow \overline{\mathbb Q}_l$, we may take the $\sigma$-companion of $\mathcal E$ to obtain a lisse $l$-adic sheaf on $U_k$, whose local monodromy around $S_k$ matches with that of the $F$-isocrystal. 

If we choose any of the four points for the parabolic structure, then we may take 
the associated elliptic curve $\tau: (C_\lambda, S')\to (X,S)$ over $\mathbb{F}_q$ to kill the -1 eigenvalues in the local monodromy. In this way we obtain ${p^2+1}$ 
$\text{GL}_2(\mathbb{Q}_\ell)$-local systems over $(C_\lambda, S')/\mathbb{F}_q,$  which just corresponds to the $(p\pm1)$-torsion points on $C_\lambda/\mathbb{F}_q.$
\begin{ques}Setup as above.
\begin{enumerate}
\item Can we intrinsically characterize the image of Equation \ref{eqn:periodic_to_l-adic}? 
\item Is there numerical evidence for the conjecture using the trace formula? More specifically, one may transform the question of counting $\text{GL}_2(\mathbb{Q}_\ell)$ local systems on $(X_1,S_1)$ with perscribed monodromy (in our case, eigenvalues of $-1$ around the parabolic point, with a non-trivial Jordan block and principal unipotent monodromy at the other punctures) to a question about counting certain types of automorphic forms via the Langlands correspondence. Drinfeld, and then Deligne-Flicker have a method to compute such numbers via the trace formula and have fully worked out this number in the case when the sheaves are supposed to have principal unipotent monodromy around each puncture \cite{drinfeld1982,deligneflicker,flicker2015}. Can we see the number $p^{2f}+1$ from the trace formula? Can we see the expected group law on the zeroes of the Higgs field from $(\mathcal{M}_\lambda)^{\text{periodic}}$ via automorphic forms?
\end{enumerate}
\end{ques}
\begin{acknowledgement*}
We thank Atsushi Shiho for explaining the fact that the functor from log $F$-isocrystals with nilpotent residues to the category of overconvergent $F$-isocrystals is fully faithful. R.K  gratefully acknowledges support from NSF Grant No. DMS-1344994 (RTG in Algebra, Algebraic Geometry and
Number Theory at the University of Georgia)
\end{acknowledgement*}

%
%\section{ideal to prove the finiteness of Hodge filtrations for general case}
%Let $(V,\nabla,\Phi)/X/K$ be a overconvergent $F$-isocrystals as in Theorem~\ref{finiteness}. We try to find some conditions such that there are only finite many Fontaine-Faltings modules $(V,\nabla,\mathrm{Fil},\Phi)/X/W$ extends $(V,\nabla,\Phi)/X/K$.
%
%The first step, we need to show there is only finite extension $(V,\nabla,\Phi)/X/W$.
%
%The second step, we need to show that for any extension $(V,\nabla,\Phi)/X/W$ there are only finitely many Hodge filtration on this extension.
%
%\begin{lem}
%	Under some conditions, there are finite modulo $p$ Hodge filtrations.
%\end{lem}
%\begin{lem} Under some conditions, every modulo $m$ Hodge fitration has unique lifting. i.e.
%	\[\mathbb H^1(\mathscr C) = 0.\] 
%\end{lem}


\begin{thebibliography}{{Ked}07}
	
	\bibitem[Abe18]{abelanglands}
	Tomoyuki Abe.
	\newblock Langlands correspondence for isocrystals and the existence of
	crystalline companions for curves.
	\newblock {\em J. Amer. Math. Soc.}, 31(4):921--1057, 2018.
	
	\bibitem[AE19]{abeesnault}
	Tomoyuki Abe and H{\'{e}}l{\`{e}}ne Esnault.
	\newblock A {L}efschetz theorem for overconvergent isocrystals with frobenius
	structure.
	\newblock {\em Ann. Scient. {\'{E}}co. Norm. Sup.}, 52(5):1243--1264, 2019.
	
	\bibitem[{Del}87]{Del87}
	P.~{Deligne}.
	\newblock {Un th\'eor\`eme de finitude pour la monodromie.}
	\newblock {\em {Discrete groups in geometry and analysis, Pap. Hon. G. D.
			Mostow 60th Birthday, Prog. Math. 67, 1-19 (1987).}}, 1987.
	
	\bibitem[DF13]{deligneflicker}
	Pierre {Deligne} and Yuval~Z. {Flicker}.
	\newblock {Counting local systems with principal unipotent local monodromy.}
	\newblock {\em {Ann. Math. (2)}}, 178(3):921--982, 2013.
	
	\bibitem[{Dri}82]{drinfeld1982}
	V.~G. {Drinfel'd}.
	\newblock {Number of two-dimensional irreducible representations of the
		fundamental group of a curve over a finite field.}
	\newblock {\em {Funct. Anal. Appl.}}, 15:294--295, 1982.
	
	\bibitem[EK16]{EK15}
	H\'el\`ene {Esnault} and Lars {Kindler}.
	\newblock {Lefschetz theorems for tamely ramified coverings.}
	\newblock {\em {Proc. Am. Math. Soc.}}, 144(12):5071--5080, 2016.
	
	\bibitem[Fal89]{Fal89}
	Gerd Faltings.
	\newblock Crystalline cohomology and {$p$}-adic {G}alois-representations.
	\newblock In {\em Algebraic analysis, geometry, and number theory ({B}altimore,
		{MD}, 1988)}, pages 25--80. Johns Hopkins Univ. Press, Baltimore, MD, 1989.
	
	\bibitem[{Fli}15]{flicker2015}
	Yuval~Z. {Flicker}.
	\newblock {Counting rank two local systems with at most one, unipotent,
		monodromy.}
	\newblock {\em {Am. J. Math.}}, 137(3):739--763, 2015.
	
	\bibitem[{Ked}07]{kedlayasemistableI}
	Kiran~S. {Kedlaya}.
	\newblock {Semistable reduction for overconvergent \(F\)-isocrystals. I:
		Unipotence and logarithmic extensions.}
	\newblock {\em {Compos. Math.}}, 143(5):1164--1212, 2007.
	
	\bibitem[Laf02]{lafforguelanglands}
	Laurent Lafforgue.
	\newblock Chtoucas de {D}rinfeld et correspondance de {L}anglands.
	\newblock {\em Invent. Math.}, 147(1):1--241, 2002.
	
	\bibitem[Lan14]{Lan14}
	Adrian Langer.
	\newblock Semistable modules over {L}ie algebroids in positive characteristic.
	\newblock {\em Doc. Math.}, 19:509--540, 2014.
	
	\bibitem[Lit18]{Li18}
	Daniel Litt.
	\newblock Arithmetic representations of fundamental groups {II}: finiteness.
	\newblock {\em arXiv preprint arXiv:1809.03524}, 2018.
	
	\bibitem[LSYZ19]{LSYZ14}
	Guitang {Lan}, Mao {Sheng}, Yanhong {Yang}, and Kang {Zuo}.
	\newblock {Uniformization of $p$-adic curves via Higgs-de Rham flows.}
	\newblock {\em {J. Reine Angew. Math.}}, 747:63--108, 2019.
	
	\bibitem[LSZ19]{LSZ13a}
	Guitang Lan, Mao Sheng, and Kang Zuo.
	\newblock Semistable {H}iggs bundles, periodic {H}iggs bundles and
	representations of algebraic fundamental groups.
	\newblock {\em J. Eur. Math. Soc. (JEMS)}, 21(10):3053--3112, 2019.
	
	\bibitem[Poo04]{poonenbertini}
	Bjorn Poonen.
	\newblock Bertini theorems over finite fields.
	\newblock {\em Ann. of Math. (2)}, 160(3):1099--1127, 2004.
	
	\bibitem[SYZ17]{SYZ17}
	Ruiran Sun, Jinbang Yang, and Kang Zuo.
	\newblock {Projective crystalline representations of {\'e}tale fundamental
		groups and twisted periodic Higgs-de Rham flow}.
	\newblock {\em arXiv preprint arXiv:1709.01485}, 2017.
	
\end{thebibliography}
 \end{document}